\documentclass[12pt,twoside,a4paper, reqno]{amsart}
\usepackage{amsthm, amsmath, amscd, amssymb,centernot}
\usepackage{amsmath,amsthm,amsfonts,amssymb}
\usepackage{amsthm}
\usepackage{graphicx}
\usepackage{hyperref}
\usepackage{color}
\usepackage[utf8]{inputenc}
\usepackage{mathtools}
\usepackage{comment}
\usepackage{fullpage}
\usepackage{csquotes}
\usepackage[T1]{fontenc}

\makeatletter
\newtheorem*{rep@theorem}{\rep@title}
\newcommand{\newreptheorem}[2]{%
\newenvironment{rep#1}[1]{%
 \def\rep@title{#2 \ref{##1}}%
 \begin{rep@theorem}}%
 {\end{rep@theorem}}}
\makeatother

\numberwithin{equation}{section}
\theoremstyle{plain}
\newtheorem{theorem}{Theorem}[section]
\newreptheorem{theorem}{Theorem}
\newtheorem{proposition}[theorem]{Proposition}

\newtheorem{corollary}[theorem]{Corollary}
\newreptheorem{corollary}{Corollary}

\theoremstyle{definition}
\newtheorem{definition}[theorem]{Definition}
\newtheorem{remark}[theorem]{Remark}
\newtheorem{notation}[theorem]{Notation}

\usepackage{tikz-cd}

  \author[Amit Kumar Singh]{Amit Kumar Singh}
\address{Department of Mathematics, SRM University AP, Amaravati 522 240, Andhra
Pradesh, India}
\email{amitks.math@gmail.com}
\title{On the smoothness of moduli spaces for quiver bundles}

 \begin{document}

 \maketitle
\begin{abstract}
In this article, we study the smoothness of the moduli space of finite quiver vector bundles over the smooth complex projective curves.
\end{abstract}
\section{Introduction}

 One of the most important and well-explored problems in algebraic geometry is to study the geometry of moduli spaces of vector bundles over algebraic curves. The moduli spaces have been a fundamental tool in the classifications of parametrized algebro-geometric objects. Note that it is not possible to parametrize the set of all isomorphism classes of vector bundles over a smooth projective curve with a fixed rank and degree in the sense of a variety or scheme. To tackle this issue, Mumford introduced the notion of (semi)stability of vector bundles over a smooth curve.  
 Further, he constructed the moduli space of stable vector bundles of the fixed rank and degree over a curve by using the geometric invariant theory (GIT) \cite{Mumford-GIT-1965}. This moduli space turns out to be a quasi projective variety. 
Later, Seshadri proved that if the rank and degree are coprime, then it is a smooth projective variety \cite{Seshadri-1982}.

The main focus of this article is to study the certain geometric properties of the moduli spaces of quiver bundles over the smooth complex projective curves. Though a considerable amount of work has been done to understand these moduli spaces (see \cite{B-S-2017, Consul-09, Consul-Prada-Schmitt-05}), but still not much is known. At this point, we must say that the study of the quivers is significant on its own. Because it has a deep connection with many branches of mathematics, specially in representation theory as a description of several interesting categories in terms of representation are given by quivers concretely.

   Let $Q = (V, A)$ be a quiver, and let $C$ be a smooth projective curve over an algebraically closed field of characteristics zero. Let $\alpha = (\alpha_i)_{i \in V} \in \mathbb{R}^{\vert V \vert}$ be a stability parameter. For a fixed tuple $t = (r_i, d_i)_{i \in V}$, the moduli space $\mathcal{M}_\alpha^s(t)$ of $\alpha$-semistable of $Q$-bundles on $C$ of type $t$ has been constructed by A. Schmitt \cite{AS-2003, AS-05}.

    For an arbitrary $n$, \'Alvarez-C\'onsul,  Garc\'ia-Prada and Schmitt have undertaken a systematic study of holomorphic $(n+1)$-chains of vector bundles \cite{Consul-Prada-Schmitt-05} (it is an example of a quiver bundle). In the same paper, they also studied the parameter region of $\alpha$ where the moduli space $\mathcal{M}_\alpha^s(t)$ may not be empty.

In this article, we prove that the moduli space $\mathcal{M}_\alpha^s(t)$ of $\alpha$-semistable of the certain quiver bundles on $C$ is smooth provided $\alpha_{at}- \alpha_{ah} \geqslant 2g -2$ by using the deformation theory and the hypercohomology technique. And we then compute the dimension of $\mathcal{M}_\alpha^s(t)$. More explicitly, at a smooth point $E_\bullet \in \mathcal{M}_\alpha^s(t)$, 
\[
\dim(\mathcal{M}_\alpha^s(t)) = 1 - \chi(E_\bullet, E_\bullet).
\]
At the end, we construct the projective bundle over $\mathcal{M}_\alpha^s (t)$ (see Theorem \ref{theorem construstion of poincare bundle}).

 \section{Preliminaries}\label{sec Preliminaries}

A quiver $Q = (V, A)$ is a directed graph consisting of the set $V$ of vertices and the set $A$ of arrows. Let $h, t:A \to V$ denote the head and tail maps of the quiver $Q$ respectively. We say the given quiver $Q= (V,A)$ is finite if the vertex set $V$ and the arrow set $A$ are finite sets. In this article, we assume all the quivers are finite. A ${Q}$-bundle $(E_{\bullet}, \phi_{\bullet})$ on a projective curve $C$ is a collection $\{E_i\  |\  i\in V\}$ of vector bundles over $C$ together with a family of morphisms  $\{\phi_a : E_{at}\to E_{ah}\ |\ a \in A\}$ of vector bundles.

Let $Q = (V, A)$ be a quiver and let $C$ be a projective curve. A $Q$-bundle $(F_\bullet, \varphi_\bullet)$ on $C$ is said to be  a sub $Q$-bundle of a $Q$-bundle $(E_\bullet, \phi_\bullet)$ on $C$, if for each $i \in V$, $F_i$ is a subbundle of $E_i$, and for each $a \in A$, $\varphi_a = {\phi_{a}}_{\vert_{F_{at}}}$.
 We write $E_{\bullet}$ a $Q$-bundle on $C$ instead of $(E_{\bullet}, \phi_{\bullet})$ unless otherwise needed. We denote the cordiality of the set $V$ by $\vert V \vert$.  

A morphism $f_\bullet$ between $Q$-bundles $(E_\bullet, \phi_\bullet)$ and $(F_\bullet, \psi_\bullet)$ consists of a family $\{f_i : E_i \to F_i\}_{i \in V}$ of morphisms of vector bundles, and for all $a \in A$, we have the following commutative diagrams
\[
\begin{tikzcd}
E_{at} \arrow{r}{\phi_a }\arrow{d}{f_{at} } & E_{ah} \arrow{d}{f_{ah}} \\
F_{at} \arrow{r}{\psi_{a}} & F_{ah}
\end{tikzcd}.
\]

\begin{definition}\label{def: degree and slope of a quiver bundle}

 Let $E_\bullet$ be a $Q$-bundle on a projective curve $C$. Let $\alpha = (\alpha_i)_{i \in V} \in \mathbb{R}^{\vert V \vert}$ be a fixed parameter.

\begin{enumerate}
\item The $\alpha$-degree of $E_\bullet$ is defined by
\begin{eqnarray*}
\deg_\alpha (E_{\bullet}) = \sum\limits_{i \in V} (\alpha_i\; \text{rk}(E_i) + \deg(E_i)).
\end{eqnarray*}
\item The $\alpha$-slope of $E_\bullet$ is defined by
\begin{eqnarray*}
\mu_\alpha (E_{\bullet}) = \frac{\sum\limits_{i \in V} (\alpha_i \; \text{rk}(E_i) + \deg(E_i))}{\sum\limits_{i \in V} \text{rk}(E_i)}.
\end{eqnarray*}  
\item The $Q$-bundle $E_\bullet$ is said to be $\alpha$-(semi)stable, if for every sub $Q$-bundle $F_\bullet$ of $E_\bullet$, 
\[
\mu_\alpha(F_\bullet) (\leqslant) \mu_\alpha(E_\bullet).
\]
\end{enumerate}
\end{definition}

 At this point, we observe that for any fixed parameter $\alpha = (\alpha_i)_{i \in V}$, a constant $c \in \mathbb{R}$ and any $Q$-bundle $E_\bullet$, $\mu_{\alpha + c}(E_\bullet) = \mu_{\alpha}(E_\bullet) + c$, where $\alpha +c = (\alpha_i + c)_{i \in V}$. Therefore, $Q$-bundle $E_\bullet$ is $\alpha$-(semi)stable if and only $Q$-bundle $E_\bullet$ is $(\alpha + c)$-(semi)stable.
 
 \begin{definition}
 Let $Q= (V, A)$ be a quiver. A pair $t = (r, d)$ is said to be a type $t$ of $Q$ if 
 $
 r = (r_i)_{i \in V} \in \mathbb{Z}_{\ge 0}^{\vert V \vert}$ \text{and} $d= (d_i)_{i \in V} \in \mathbb{Z}^{\vert V \vert}$.
\end{definition}
 For a given a $Q$-bundle $E_\bullet =  \{ E_i \mid i  \in V \}$, its type  $t$ of the quiver $Q$ is defined as
 $t = (\text{rk}(E_\bullet), \deg(E_\bullet))$,
 where $\text{rk}(E_\bullet) = (\text{rk}(E_i))_{i \in V}$, $\deg(E_\bullet) = (\deg(E_i))_{i \in V}$.
 \begin{remark}
 We fix the type $t = (r, d)$ of a quiver $Q=(V,\; A)$, where $ r = (r_i)_{i \in V}$ and $d = (d_i)_{i \in V}$.  Note that the definition of $\alpha$-degree and $\alpha$-slope of a quiver bundle (Definition \ref{def: degree and slope of a quiver bundle}) only depend on its type. Therefore, the $\alpha$-degree and $\alpha$-slope of a type $t = (r, d)$ defined as 
 \begin{eqnarray*}
 \deg_\alpha(t) = \sum\limits_{i \in V}(\alpha_i \; r_i + d_i)
\end{eqnarray*} 
and
 \begin{eqnarray*}
 \mu_{\alpha}(t) = \frac{\deg_\alpha (t)}{\sum\limits_{i \in V} r_i}.
 \end{eqnarray*}
 \end{remark}
 
 For fixed type $t = (r, d)$ of a quiver $Q$, let $\mathcal{M}_{\alpha}^s (t)$ be the set of S-equivalence classes of $\alpha$-semistable quiver bundles of type $t$. The GIT construction of this space is done in \cite{Consul-09}.

 A holomorphic $(n+1)$-chain of vector bundles over a curve $C$ is a collection $\{ E_i \; \vert \; i = 1, 2, \dots , n +1 \}$ of vector bundles over $C$ with the morphisms $\{ \varphi_i : E_i \to E_{i+1} \; \vert \; i = 1, 2, \dots, n \}$ (it is an example of a quiver bundle). Moreover, for holomorphic $(n+1)$-chain of vector bundles, the moduli space $\mathcal{M}_\alpha^s(t)$ is smooth, if $\alpha_i - \alpha_{i-1} \geqslant 2g -2$, for all $1 \leqslant i \leqslant n+1$ (See \cite[Theorem 3.8(vi)]{Consul-Prada-Schmitt-05}, \cite[Theorem 10]{Consul-09}).   
 
We look for similar conditions on the parameter $\alpha$ establishing the smoothness of $\mathcal{M}_\alpha^s(t)$ for quiver $Q$ (see Theorem \ref{main theorem-1}).

\section{Hyper-cohomology for quiver bundles}
In this section and the rest we assume the quivers $Q=(V, A)$ to have the following property: between any two vertices $i_1$ and $i_2$, there exist atmost one arrow.
For a fixed stability parameter $\alpha = (\alpha_i)_{i \in V}$, let $E_\bullet =(E_i)_{i \in V}$ and $E'_\bullet =(E'_i)_{i \in V}$ be $Q$-bundles, and let $f_\bullet$ be a $Q$-morphism between $E_\bullet$ and $E'_\bullet$. We assume that the curves $C$ are the smooth complex projective curves unless otherwise specified. 

We first analyse $\text{Ext}^1(E_\bullet, E'_\bullet)$ by using the hyper-cohomology groups $\mathbb{H}^i(F^\bullet (E_\bullet, E'_\bullet))$ of a 2-step complex of vector bundles
\[
F^\bullet (E_\bullet, E'_\bullet) : F^0 \rightarrow F^1.
\]
This complex has terms
\begin{align*}
F^0 = \bigoplus_{i \in V} \text{Hom}({E}_i , E'_i) \quad \text{and} \quad F^1 = \bigoplus_{a \in A} \text{Hom}({E}_{at} , E'_{ah})
\end{align*}
and differential 
\[
b(\psi_i)_{i \in V} = \sum\limits_{a \in A} b_a (\psi_{ah}, \psi_{at}), \quad \text{for} \quad \psi_i \in \text{Hom}(E_i, E'_i),
\]
where $b_a : \text{Hom}(E_{ah}, E'_{ah}) \oplus \text{Hom}(E_{at}, E'_{at}) \rightarrow \text{Hom}(E_{at}, E'_{ah}) \hookrightarrow F^1$ is defined by \[ b_a(\psi_{ah}, \psi_{at}) = \psi_{ah} \circ \phi_a - \phi_a \circ \psi_{at}.\]

\noindent
By applying the cohomology functor to this complex of vector bundles, we get the linear maps of vector spaces
\begin{eqnarray*}
d = H^p(b) : H^p(F^0) \rightarrow H^p(F^1), 
\end{eqnarray*} 
for $p = 0,1$, where 
\[
H^p (F^0) = \bigoplus_{i \in V} \text{Ext}_{C}^p(E_i, E'_i), \quad H^p (F^1) = \bigoplus_{a \in A} \text{Ext}_{C}^p(E_{at}, E'_{ah})
\]
and 
\begin{align*}
d((\psi_i)_{i \in V}) = \sum\limits_{a \in A} d_a(\psi_{at}, \psi_{ah}), \quad \text{for} \quad \psi_{i} \in \text{Ext}_C^i (E_i, E'_i).
\end{align*}
Here, the map
\begin{eqnarray*}
d_a : \text{Ext}_C^p (E_{ah}, E'_{ah}) \oplus \text{Ext}_C^p(E_{at}, E'_{at})  \rightarrow \text{Ext}_C^p (E_{at}, E'_{ah}) \hookrightarrow H^0(F^1)
\end{eqnarray*}
is given by
\[
d_a (\psi_{ah}, \psi_{at}) = \psi_{ah} \circ {\phi'}_a - \phi_a \circ \psi_{at}.
\]

\begin{proposition}\label{proposition of isomorphism on hypercohomology}
There are natural isomorphisms
\begin{eqnarray*}
\text{Hom}(E_\bullet, E'_\bullet) \cong \mathbb{H}^0(F^\bullet(E_\bullet, E'_\bullet)),
\end{eqnarray*}
\begin{eqnarray*}
\text{Ext}^1(E_\bullet, E'_\bullet) \cong \mathbb{H}^1(F^\bullet(E_\bullet, E'_\bullet)),
\end{eqnarray*}
and an exact sequence
\begin{alignat*}{2}
        0 &\rightarrow  \mathbb{H}^0(F^\bullet(E_\bullet, E'_\bullet)) \rightarrow H^{0}(F^0) \rightarrow H^{0}(F^1) \rightarrow \mathbb{H}^1(F^\bullet(E_\bullet, E'_\bullet)) \\
        &\rightarrow H^{1}(F^0)\xrightarrow{d}  H^1(F^1) \rightarrow \mathbb{H}^2(F^\bullet(E_\bullet, E'_\bullet)) \rightarrow 0.
    \end{alignat*}  
\end{proposition}
\begin{proof}
 This follows form \cite[Theorem 4.1 and Theorem 5.1]{Gothen-King-2004}.
\end{proof}
For vector spaces $E$ and $E'$, we define 
\begin{eqnarray*}
h^i(E, E')= \dim (\text{Ext}^i_C(E, E')), \quad \chi(E, E') = h^0(E, E') - h^1(E, E').
\end{eqnarray*}
In similar way, for $Q$-quiver bundles $E_\bullet$, $E'_\bullet$, we define
\begin{eqnarray*}
h^i (E_\bullet, {E'}_{\bullet}) = \dim (\mathbb{H}^i (E_\bullet, E'_\bullet)), \quad \chi(E_\bullet, E'_\bullet) = h^0(E_\bullet, E'_\bullet) - h^1(E_\bullet, E'_\bullet) + h^2(E_\bullet, E'_\bullet).
\end{eqnarray*}

\begin{notation} Let $Q= (V, A)$ be a quiver. For  $Q$-bundles $E_\bullet = (E_i)_{i \in V}$ and  $E'_\bullet = (E'_i)_{i \in V}$ on a curve $C$, we denote the following notations: 
\[ r_i = \text{rk}(E_i), \quad {r'}_i = \text{rk}(E'_i), \quad \deg(E_i) = d_i \quad \text{and} \quad \deg(E'_i) = d'_i. \]
 \end{notation}
 \begin{proposition}\label{prop.4} Let $C$ be a complex projective curve with genus $g$. Let $E_\bullet= (E_i)_{i \in V}$ be a $Q$-bundles on $C$. Then
 \begin{eqnarray*}
 \begin{split}
 \chi(E_\bullet, E'_\bullet) & = \sum\limits_{i \in Q} \chi(E_i, E'_i) - \sum\limits_{a \in A}\chi(E_{at}, E'_{ah})\\
 & = (1-g) (\sum\limits_{i \in V} r_i r'_i - \sum\limits_{a \in A} r_{at}r'_{ah})  \\
 & \quad + \sum\limits_{i \in V} (r'_i d_i - r_i d'_i) + \sum\limits_{ a \in A} (r'_{at} d_{ah} - r_{ah} d'_{at}).
 \end{split}
 \end{eqnarray*}
\begin{proof}
From the long exact sequence in Proposition 2.1, we get 

\begin{eqnarray*}
 \chi(E_\bullet, E'_\bullet) & = \sum\limits_{i \in V} \chi(E_i, E'_i) - \sum\limits_{a \in A}\chi(E_{at}, E'_{ah}).\\
 \end{eqnarray*}
 The Riemann Roch formula
 \[\chi(E,F) = (1-g)\textrm{rk}(E)\textrm{rk}(F) + \textrm{rk}(E)\textrm{deg}(F) - \textrm{rk}(F)\textrm{deg}(E)\]
implies the second equality.
\end{proof}
 \end{proposition}
 
 \begin{corollary}
 For any extension $$0\to E'\to E \to E^{''} \to 0$$ of $Q$-bundles, 
 
 \[\chi(E, E) = \chi(E', E') + \chi(E^{''}, E^{''}) + \chi(E^{''}, E') + \chi(E', E^{''}).\]
 \end{corollary}
In the next result, we investigate the vanishing criteria of $\mathbb{H}^0$ and $\mathbb{H}^2$ by putting some restrictions on the stability parameters $\alpha$. 
 
 \begin{proposition}
  Suppose $E_\bullet$ and $E'_\bullet$ are $\alpha$-semistable quiver bundles.
  \begin{enumerate}
      \item If $\mu_\alpha(E'_\bullet) < \mu_\alpha(E_\bullet)$, then $\mathbb{H}^0(F^\bullet(E_\bullet, E'_\bullet)) = 0.$
      \item If $\mu_\alpha(E'_\bullet) = \mu_\alpha(E_\bullet)$ and $E_\bullet$ is $\alpha$-stable, then
      \begin{equation*}
   \mathbb{H}^0(F^\bullet(E_\bullet, E'_\bullet))     = \begin{cases}
         \mathbb{C}, \quad \text{if} \quad  E_\bullet \cong E'_\bullet  \\
         0 \quad \quad \text{otherwise} 
        \end{cases}.
  \end{equation*}

  \end{enumerate}
 \end{proposition}
 
\begin{proposition}\label{prop.1}
 Let $E_\bullet$, $E'_\bullet$ be $Q$-bundles over $C$ and let $D \subset A$. For each $a \in D$, let $\epsilon_a \geqslant 0$. 
 \begin{enumerate}
 \item Suppose that the following conditions are satisfied:
  \begin{enumerate}
 \item[(a)] $E_\bullet$, $E'_\bullet$ are $\alpha$-semistable with $\mu_\alpha(E_\bullet)= \mu_\alpha(E'_\bullet)$,
 \item[(b)] For all $a \in A \setminus D$, $\alpha_{at} - \alpha_{ah} > 2g -2$,
 \item[(c)] For all $a \in D$, one of $E_\bullet$, $E'_\bullet$ is $\alpha - \epsilon_{at} u_{at}$-stable and $\alpha_{at} - \alpha_{ah} \geqslant 2g -2$. 
\end{enumerate}  
Then $\mathbb{H}^2(F^\bullet (E_\bullet, E'_\bullet)) = 0$.
\item If $E_\bullet$, $E'_\bullet$ are $\alpha$-semistable with the same $\alpha$-slope, and $\alpha_{at} - \alpha_{ah} > 2g -2$ for all $a \in A$, then $\mathbb{H}^2(F^\bullet (E_\bullet, E'_\bullet)) = 0$.

\item If one of $E_\bullet$, $E'_\bullet$ is $\alpha$-stable and the other one is $\alpha$-semistable with the same $\alpha$-slope, and $\alpha_{at} - \alpha_{ah} \geqslant 2g -2$, for all $a \in A$, then $\mathbb{H}^2(F^\bullet (E_\bullet, E'_\bullet)) = 0$.
\item If for all $a \in A$, $\phi_a$ is injective, or $\phi_a$ is generically surjective, then $\mathbb{H}^2(F^\bullet (E_\bullet, E'_\bullet)) = 0$.

 \end{enumerate}
 \end{proposition}
 
 \begin{proof}
 From the exact sequence in Proposition \ref{proposition of isomorphism on hypercohomology}, $\mathbb{H}^2(F^\bullet (E_\bullet, E'_\bullet)) = 0$ if and only if the map $d: H^1(F^0) \to H^1(F^1)$ is surjective if and only if $d_a$ is surjective for all $a \in A$. Using Serre duality, $d_a$ is surjective if and only if the map is given by 
 \[
 P_a: Hom_C(E'_{ah}, E_{at}\otimes K)  \to Hom_C(E'_{ah}, E_{ah}\otimes K) \oplus Hom_C(E_{at}, E'_{at}\otimes K)
 \]
 \noindent
 given by $P_a(\xi_a) = ((\phi_a \otimes id_K)\circ \xi_a, \xi_a\circ \phi'_a)$ is injective, where $\xi_a \in Hom_X(E'_{ah}, E_{at}\otimes K)$.
 
  \noindent
 Let $a \in A$ and $\xi_a: E'_{ah} \to E_{at}\otimes K$ be in $\text{ker}(P_a)$. For simplicity, we denote $\text{Im} (\xi_a)$ and $\text{ker}(\xi_a)$ by $I_{at}$ and $N_{ah}$ respectively. Note that
 $
I_{at} = \text{Im} (\xi_a) \subset E_{at}\otimes K$ and $N_{ah} = \text{ker}(\xi_a) \subset E'_{ah}$.
Then the fact that $\xi_a \in \text{ker} (P_a)$ is equivalent to the fact that the maps
  
\[
 E'_{ah} \xrightarrow{\xi_a} E_{at} \otimes K \xrightarrow{\phi_{a} \otimes \text{id}} E_{ah}\otimes K \quad\textrm{and} \quad E'_{at} \xrightarrow{\phi'_a} E'_{ah} \xrightarrow{\xi_a} E_{at}\otimes K
 \] 
 are both zero. Therefore,

\begin{itemize}
    \item The first map is zero if and only if $\text{Im}(\xi_a) = I_{at} \subset \text{ker} (\phi_a \otimes id)$ if and only if $I_{at} \otimes K^* \subset \text{ker}(\phi_a)$. So in this case, we have the following diagram  
       \[
       \begin{tikzcd}
    I_{at} \otimes K^* \arrow{r}\arrow[hookrightarrow]{d}  & 0 \arrow[hookrightarrow]{d} \\
    E_{at} \arrow{r}{\phi_{a}} & E_{ah}
    \end{tikzcd}.
    \]
Thus, it defines a sub $Q$-bundle $E_{\bullet I_{at}}$ of $E_\bullet$ by

\begin{equation*}
       (E_{\bullet I_{at}})_{i \in V} = 
       \begin{cases}
         I_{at} \otimes K^* \quad \text{if} \quad i = at \\
         0 \quad \quad \quad \quad  \; \text{if} \quad i \neq at
        \end{cases}.
  \end{equation*}

    \item The second map is zero if and only if $\text{Im} ({\phi}'_{a}) \subset N_{ah} = \text{ker} (\xi_a)$. That is, in this case we have the following diagram  
    \[
\begin{tikzcd}
E'_{at} \arrow{r}{\phi'_a }\arrow[equal]{d} & N_{ah}\arrow[hookrightarrow]{d} \\
E'_{at} \arrow{r}{\phi'_{a}} & E'_{ah}
\end{tikzcd}.
\]

\noindent
We therefore define a sub $Q$-bundle $E'_{\bullet N_{ah}}$ of $E'_\bullet$ as

\begin{equation*}
       (E'_{\bullet N_{ah}})_{i \in V} = \begin{cases}
         N_{ah} \quad \text{if} \quad i = ah \\
         E'_i \quad \;\; \text{if} \quad i \neq ah
        \end{cases}.
  \end{equation*}

\end{itemize}

\noindent
Let $k_{ah} = \textrm{rk}(N_{ah})$, $l_{ah} = \deg(N_{ah})$. Then

\[\mu_\alpha (E'_{\bullet N_{ah}}) = \frac{l_{ah} + \sum_{j \neq ah}d'_j + (\alpha_{ah} k_{ah} + \sum_{j \neq ah} \alpha_j r'_j)}{k_{ah} + \sum_{j \neq ah}r'_j}.\]
\noindent 
From the short exact sequence 
\[
0 \to N_{ah} \to E'_{ah} \to I_{at} \to 0, 
\]
we see that 
\[
\textrm{rk}(I_{at}\otimes K^{*}) = \textrm{rk}(I_{at}) = r'_{ah} - k_{ah}
\]and 
\[
\textrm{deg}(I_{at}\otimes K^{*}) = \textrm{deg}(I_{at}) + \textrm{deg}(K^{*})\textrm{rk}(I_{at}) = d'_{ah} - l_{ah} + (2-2g)(r'_{ah} - k_{ah}).
\]
Hence,
 \begin{eqnarray*}
 \begin{split}
 \mu_\alpha (E'_{\bullet I_{at}}) & = \frac{\textrm{deg}(I_{at}\otimes K^{*}) + \alpha_{at}\textrm{rk}(I_{at}\otimes K^{*})}{\textrm{rk}(I_{at}\otimes K^{*})} \\
& = \frac{d'_{ah} - l_{ah}}{r'_{ah} - k_{ah}} + 2-2g + \alpha_{at}.
 \end{split}
 \end{eqnarray*}
Using the formula for $\mu_\alpha (E'_{\bullet I_{at}})$ and $\mu_\alpha (E'_{\bullet N_{ah}})$, 
we obtain
\begin{eqnarray}\label{relneqn}
\begin{split}
   & (k_{ah} + \sum_{j \neq ah}r_j)\mu_\alpha (E'_{\bullet N_{ah}}) + (r'_{ah} - k_{ah})\mu_\alpha (E_{\bullet I_{at}}) = \\ 
& \sum_{i \in V}d'_i + (r'_{ah} - k_{ah})(2-2g) + (\alpha_{ah}-\alpha_{at})k_{ah} + (\sum_{j \neq ah}\alpha_j r'_j + \alpha_{at}r'_{ah}).
\end{split}
\end{eqnarray}
    
For all $a \in A$, to prove $P_a$ is injective. In other words, to prove $\textrm{ker}(P_a) = 0$, it suffices to show that for arbitrary element $\xi_a \in \textrm{ker}(P_a)$,  $\xi_a = 0$.

    $\text{Proof of}\; (1)$. First we prove for $a \in A \setminus D$, so $\alpha_{at}- \alpha_{ah} > 2g - 2$. 
Since $E_{\bullet}$ and $E'_{\bullet}$ are $\alpha$-semistable, we have $\mu_{\alpha} (E'_{\bullet N_{ah}}) \leq \mu_{\alpha} (E'_{\bullet})$, $\mu_{\alpha} (E_{\bullet I_{at}}) \leq \mu_{\alpha} (E_{\bullet})$ and $\mu_{\alpha}(E_{\bullet}) = \mu_{\alpha}(E'_{\bullet})$, therefore now \eqref{relneqn} is smaller than and equal to

\begin{multline*}
    (k_{ah} + \sum_{j \neq ah} r'_j)\mu_{\alpha}(E'_{\bullet}) + (r'_{ah} - k_{ah})\mu_{\alpha}(E_{\bullet}) = (\sum_{j \in V}r'_{j})\mu_{\alpha}(E'_{\bullet}) = \sum_{j \in V} d'_j + \sum_{j \in V}\alpha_j r'_j. 
\end{multline*}
Thus we get 

\begin{equation*}
\sum_{j \in V}d'_j + (r'_{ah} - k_{ah})(2-2g) + (\sum_{j \neq ah}\alpha_j r'_j + \alpha_{at}r'_{ah}) \leq \sum_{j \in V} d'_j + \sum_{j \in V}\alpha_j r'_j 
\end{equation*} 
if and only if
\begin{eqnarray}\label{leqn-2}
(\alpha_{at} - \alpha_{ah})(r'_{ah} - k_{ah}) \leq (r'_{ah} - k_{ah})(2g-2)
\end{eqnarray}
Now, for $a \in A \setminus D$, let $\xi_a \neq 0$. Then $\textrm{rk}(I_{at}) > 0$. This implies that $r'_{ah} - k_{ah} > 0$. Thus the inequality \ref{leqn-2} gives $\alpha_{at} - \alpha_{ah} \leq 2g-2$ which contradicts our assumption that $\alpha_{at}- \alpha_{ah} > 2g - 2$. Therefore,  $\xi_a = 0$ for all $a \in A \setminus D$.

Now let $a \in D$. Then $E_{\bullet}$ is $\alpha_{\bullet}$-stable where $\alpha_a = \alpha + \epsilon_{at}u_{at}$. As $E_{\bullet I_{at}}$ is a proper quiver subbundle of $E_{\bullet}$, we get $\mu_{\alpha_a}(E_{\bullet I_{at}}) < \mu_{\alpha_a}(E_{\bullet})$.
Define $\lambda_a(E_{\bullet}) = \frac{r_{at}}{\sum_{j \in V}r_j}$. The strict inequality
\[
\mu_{\alpha_a}(E_{\bullet I_{at}}) < \mu_{\alpha_a}(E_{\bullet})
\]
gives the inequality
\begin{equation*}
    \mu_{\alpha}(E_{\bullet I_{at}}) + \epsilon_{at} < \mu_{\alpha}(E_{\bullet}) + \frac{\epsilon_{at}r_{at}}{\sum_{j \in V}r_j}.
\end{equation*}
This is equivalent to 
\begin{eqnarray*}
    \mu_{\alpha}(E_{\bullet I_{at}}) - \mu_{\alpha}(E_{\bullet}) < \epsilon_{at}\lambda_a(E_{\bullet}) - \epsilon_{at}.
\end{eqnarray*}
Therefore, $\mu_{\alpha}(E_{\bullet I_{at}}) < \mu_{\alpha}(E_{\bullet})$ as $\epsilon_at \geq 0$. By using this strict inequality, (\ref{leqn-2}) becomes $\alpha_{at} - \alpha _{ah} < 2g-2$. As in the case of $a \in A \setminus D$, this yields contradiction as for all $a \in D$ we have $\alpha_{at} - \alpha _{ah} \geq 2g -2 $. This establishes $\xi_a = 0$ for all $a \in D$ and thus that $\textrm{ker}(P_a)  = 0$ for all $a \in D$.

    $\text{Proof of}\; (2). \;$ For $D = \emptyset$, we see that (1) implies (2).

     $\text{Proof of}\; (3). \;$ For $D = A$ and all $\epsilon_{at} = 0$ ($a \in A$), we see that (1) implies (3).

 \end{proof}

 \begin{corollary}\label{corollary-1}
 Suppose that $E_\bullet$, $E'_\bullet$ are $\alpha$-semistable $\mathcal{Q}$-quiver with the same $\alpha$-slope. Let $\alpha_{at} - \alpha_{ah} > 2g -2$ for all $a \in A$. Then
 \begin{eqnarray*}
 \dim \text{Ext}^1(E_\bullet, E'_\bullet) = h^0 (E_\bullet), E'_\bullet) - \chi (E_\bullet), E'_\bullet).
\end{eqnarray*}   
 \end{corollary}
 
  \section{Deformation theory and smoothness conditions}
In this section, we prove the main result (Theorem \ref{main theorem-1}) of the article with the suitable conditions on the parameters by using the deformation theory.

Let $\mathbb{C}[\varepsilon]/(\varepsilon^2)$ be the ring of dual numbers over $\mathbb{C}$ and denoted it by $\mathbb{C}[\varepsilon]$. We use the symbol $C_\varepsilon$ for denoting the fiber product $C \times \text{Spec}(\mathbb{C}[\varepsilon])$. Let $p_\varepsilon : C_\varepsilon \rightarrow C$ be the first projection morphism, and $E$ be a vector bundle on $C$. Then $E[\varepsilon] := p^*_\varepsilon (E)$. 

\begin{definition}\label{def of infinitesimal def on quiver}
Let $Q=(V, A)$ be a quiver. A linear infinitesimal deformation of a $Q$-bundle $(E_\bullet, \phi_\bullet)$ is a pair $(E_{\bullet \varepsilon}, \phi_{\bullet \varepsilon})$, where for each $i \in V$, $E_{i \varepsilon}$ is a vector bundle on $C_\varepsilon$, and for each $a \in A$, $\phi_{a \varepsilon} : E_{at \varepsilon} \rightarrow E_{ah \varepsilon}$ is a morphism, and for each $i \in V$, $E_i \cong E_{i \varepsilon} \otimes_{\mathbb{C}[\varepsilon]} \mathbb{C}$ such that $\phi_{a\varepsilon} \otimes_{\mathbb{C}[\varepsilon]} \mathbb{C}$ goes over into $\phi_a$.
\end{definition}
An isomorphism two between infinitesimal deformations $(E'_{\bullet \varepsilon}, \phi'_{\bullet\varepsilon})$ is defined as an isomorphism $\lambda_{i \varepsilon} : E'_{i \varepsilon} \rightarrow E''_{i \varepsilon}$ restricting to the identity on $E$ such that the following diagram is commutative:

\[
\begin{tikzcd}[column sep={15.5em,between origins},row sep=6em]
E'_{at \varepsilon} \arrow{r}{\phi'_{a \varepsilon}} \arrow[swap]{d}{\lambda_{at \varepsilon}} & E'_{ah \varepsilon} \arrow{d}{\lambda_{ah \varepsilon}} \\
E''_{at \varepsilon} \arrow{r}{\phi''_{a \varepsilon}}  & E''_{ah \varepsilon} \
\end{tikzcd}.
\]
 
        Consider the following complex of $Q$-bundles on $C$. For a short notation, it is denoted by $[\cdot, \;\phi_\bullet]$:
        $
         F^0 \xrightarrow{[\cdot,\; \phi_\bullet]} F^1
        $,
        where for every section $\alpha_i$ on $F^0$ over an open set $U \subset C$, we have $\alpha_{at} \circ \phi_a - \phi_a \circ \alpha_{ah}$. 
        
        \begin{proposition}\label{prop.2}
        The isomorphism classes of linear infinitesimal deformations of the pair $(E_\bullet, \phi_\bullet)$ are canonically parametrized by the first hypercohomology group $\mathbb{H}^1([\cdot, \phi_\bullet])$ of the complex $[\cdot, \phi_\bullet]$. 
        
        \begin{proof}
        Let $(E_{\bullet \varepsilon}, \phi_{\bullet \varepsilon})$ be an infinitesimal deformation of $(E_\bullet, \phi_\bullet)$. Consider open affine covering $\mathcal{U} = (\mathcal{U}_{\alpha} = \text{Spec}A_\alpha)_{\alpha \in I}$  of $C$. For $\alpha, \beta \in I$, we write $\mathcal{U}_{\alpha \beta} = \mathcal{U}_\alpha \cap \mathcal{U}_\beta = \text{Spec} A_{\alpha \beta}$. Then $\mathcal{U}_\alpha[\varepsilon] = \text{Spec} A_\alpha [\varepsilon]$ is an open affine covering of $C_\varepsilon$. For each $i \in V$, we denote $H^0 (\mathcal{U}_\alpha, E_i)$ and $H^0 (\mathcal{U}_{\alpha \beta}, E_i)$ by $M_{i \alpha}$ and $M_{i \alpha \beta}$ respectively.
        
        As $E_{i \varepsilon}$ is a vector bundle on $C_\varepsilon$, for every $\mathcal{U}_\alpha$, we have an isomorphism $f_{i \alpha} : {E_{i \varepsilon}}_{\vert_{\mathcal{U}_\alpha[\varepsilon]}} \xrightarrow{\sim} \widetilde{M_{i \alpha}}[\epsilon]$. Therefore, for the open sets $\mathcal{U}_{\alpha \beta}$, it induces the isomorphism ${f_{i \alpha}}_{\vert_{\mathcal{U}_{\alpha \beta}[\varepsilon]}} \circ {f^{-1}_{i \beta}}_{\vert_{\mathcal{U}_{\alpha \beta}[\varepsilon]}} : \widetilde{M_{i \beta \alpha}}[\varepsilon] \rightarrow \widetilde{M_{i \alpha \beta}}[\varepsilon]$, and that corresponds to an element in $\oplus_{a \in A}(\text{Hom} (\widetilde{M_{ah \alpha \beta}}[\varepsilon], \widetilde{M_{at \alpha \beta}}[\varepsilon]))$ of the form $1 + \varepsilon \oplus_{a \in A}\zeta_{a \alpha \beta}$. Moreover,  $\oplus_{a \in A}\zeta'_{a \alpha \beta} +\oplus_{a \in A}\zeta''_{a \alpha \beta}$ corresponds to the composite automorphism $(1 + \varepsilon \oplus_{a \in A}\zeta'_{a \alpha \beta}) \circ (1 + \varepsilon \oplus_{a \in A}\zeta''_{a \alpha \beta}) = 1 + \varepsilon \oplus_{a \in A}(\zeta'_{i \alpha \beta} + \zeta''_{i \alpha \beta})$, since $\varepsilon^2 = 0$. Now, one can see that the element $\zeta_{a \alpha \beta}$ defines a 1-cocycle for $\text{Hom}(E_{ah}, E_{at})$, and hence it gives an element of $H^1(F^1)$.
        
        Conversely, given a 1-cocycle $\zeta_{a \alpha \beta}$ with values in $(\text{Hom}(E_{ah}, E_{at}))_{a \in A}$ the corresponding sheaf $\text{Hom}(E_{ah \varepsilon}, E_{ah \varepsilon})$ may be determined by gluing sheaves $\widetilde{M_{i \alpha}}[\varepsilon]$ and $\widetilde{M_{i \beta}}[\varepsilon]$ with the open set $\mathcal{U}_{\alpha \beta}$ by the isomorphism 
        $ {\widetilde{M_{i \alpha}}[\varepsilon]}_{\vert_{U_{\alpha \beta}[\varepsilon]}} \rightarrow {\widetilde{M_{i \beta}}[\varepsilon]}_{\vert_{U_{\alpha \beta}[\varepsilon]}}$.

        Note that the infinitesimal deformation $(E_{\bullet \varepsilon}, \phi_{\bullet \varepsilon})$ of $(E_\bullet, \phi_\bullet)$, $\phi_{\bullet \varepsilon}$ may be described by giving homomorphisms 
        \[
        \phi_{a \varepsilon \alpha} : M_{at \alpha}[\varepsilon] \rightarrow M_{ah \alpha}[\varepsilon]
        \]
        which restrict to $\phi_a$ modulo $\varepsilon$ and that are compatible with gluing the isomorphisms.

        In other words, this means that $\phi_{a \varepsilon \alpha} = \phi_{a} + \varepsilon \mu_{a \alpha}$, for some homomorphism $\mu_{a \alpha}:M_{at \alpha} \rightarrow M_{ah \alpha}$, for each $\alpha, \beta$, we have following commutative diagram 
        
\[
\begin{tikzcd}[column sep={15.5em,between origins},row sep=6em]
M_{at \alpha \beta} [\varepsilon] \arrow{r} \arrow[swap]{d}{1 + \varepsilon \zeta_{a \alpha \beta}} & M_{at \alpha \beta} [\varepsilon] \arrow{d}{1 + \varepsilon \zeta_{a \alpha \beta}} \\
M_{ah \alpha \beta} [\varepsilon] \arrow{r}  & M_{ah \alpha \beta} [\varepsilon]
\end{tikzcd}.
\]
By replacing the expressions of $\phi_{\varepsilon \alpha \beta}$ given above, it follows from definition \ref{def of infinitesimal def on quiver} that

\[
(\oplus_{a \in A}\lambda_{a \beta} - \oplus_{a \in A}\lambda_{a \alpha})_{|_{\mathcal{U_{\alpha \beta}[\varepsilon]}}} = [\oplus_{a \in A}\zeta_{a \alpha \beta}, \phi_\bullet]
\] 
   
   If $(E'_{\bullet \varepsilon}, \phi'_{\bullet \varepsilon})$ is another infinitesimal deformation of $(E_\varepsilon, \phi_\varepsilon)$ is isomorphic to $(E_{\bullet \varepsilon}, \phi_{\bullet \varepsilon})$ and if $(\lbrace \mu'_{a \alpha} \rbrace, \lbrace \zeta'_{i \alpha \beta} \rbrace)$ is the pair constructed from $(E'_{\bullet \varepsilon}, \phi'_{\bullet \varepsilon})$ as $(\lbrace \mu_{a \alpha} \rbrace, \lbrace \zeta_{i \alpha \beta} \rbrace)$ from $(E_{\bullet \varepsilon}, \phi_{\bullet \varepsilon})$, then it follows that 
   
   \[
   \mu'_{a \alpha} = \mu_{a \alpha} +[\lambda_{a \alpha}, \phi_a] \quad \text{and} \quad \zeta'_{a \alpha \beta} = \zeta_{a \alpha \beta} + (\lambda _{ a \beta} - \lambda_{a \alpha}).
   \]

   Now, let us take the same open covering $\mathcal{U} = (\mathcal{U}_\alpha)$ of the curve $C$ as a \v{C}ech covering to calculate the cohomology, then the pair $(\lbrace \lambda_{a \alpha} \rbrace, \lbrace \zeta_{a \alpha \beta} \rbrace)$ associated to $[\cdot, \phi_\bullet]$ defines an element of $\mathbb{H}^1([\cdot, \phi_\bullet])$ of the complex $[\cdot, \phi_\bullet]$ that depends only on the isomorphism classes of infinitesimal deformations $(E_{\bullet \varepsilon}, \phi_{\bullet \varepsilon})$ of $(E_{\bullet}, \phi_{\bullet})$. This proves the result.
        \end{proof}
        \end{proposition}
        
 \begin{corollary}\label{prop.3}
 The tangent space $T_{(E_{\bullet}, \phi_{\bullet})} \mathcal{M}_\alpha^s(t)$ to $\mathcal{M}_\alpha^s(t)$ at the point $(E_{\bullet}, \phi_{\bullet})$ is canonically isomorphic to $\mathbb{H}^1([\cdot, \phi_\bullet])$.
 \end{corollary}
 \begin{remark}
Passing to cohomology of the complex $[\cdot, \phi_\bullet]$, we have
\begin{alignat*}{2}
        0 &\rightarrow  \mathbb{H}^0([\cdot, \phi_\bullet]) \rightarrow H^{0}(F^0) \xrightarrow{[\cdot, \phi_\bullet]} H^{0}(F^1) \rightarrow \mathbb{H}^1([\cdot, \phi_\bullet]) \\
        &\rightarrow H^{1}(F^0)\xrightarrow{[\cdot, \phi_\bullet]}  H^1(F^1) \rightarrow \mathbb{H}^2([\cdot, \phi_\bullet]) \rightarrow 0.
    \end{alignat*}  
    \end{remark} 
          Now, we therefore conclude the following main result of the article.

 \begin{theorem}\label{main theorem-1}
 Let $E_\bullet$ be a $\alpha$-stable quiver bundle of type t on a quiver $Q$.
 \begin{enumerate}
     \item The Zariski tangent space at the point defined by $E_\bullet$ in the moduli space $\mathcal{M}_\alpha^s (t)$ is isomorphic to $\mathbb{H}^1(F^\bullet(E_\bullet, E_\bullet)).$
     \item If $\mathbb{H}^2(F^\bullet(E_\bullet, E_\bullet)) = 0$, then the moduli space $\mathcal{M}_\alpha^s (t)$ is smooth in a neighbourhood of the point defined by $E_\bullet$.
     \item $\mathbb{H}^2(F^\bullet(E_\bullet, E_\bullet)) = 0$ if and only if the homomorphism (in Proposition \ref{prop.1})
     \[
     d : \bigoplus_{i \in V} \text{Ext}^1(E_i, E_i) \rightarrow \bigoplus_{a \in A}\text{Ext}^1(E_{at}, E_{ah})
     \]
     is onto.

     \item At a smooth point $E_\bullet \in \mathcal{M}_\alpha^s (t)$, the dimension of the moduli space $\mathcal{M}_\alpha^s (t)$ is 
     
     \[\dim \mathcal{M}_\alpha^s (t) = h^1(E_\bullet, E_\bullet) = 1 - \chi(E_\bullet, E_\bullet).\]
     \item If for each arrow $a \in A$, $\phi_a: E_{ta} \to E_{ha}$ is injective or generically surjective, then $\mathbb{H}^2(F^\bullet(E_\bullet, E_\bullet)) = 0$ and therefore $E_\bullet$ defines a smooth point of $\mathcal{M}_\alpha^s (t)$.
     
     \item If $\alpha_{at} - \alpha_{ah} \geq 2g -2$ for all $a \in A$, then $E_\bullet$ defines a smooth point in the moduli space, and hence $\mathcal{M}_\alpha^s (t)$ is smooth.
 \end{enumerate}
 \begin{proof}
The result (1) follows from Corollary \ref{prop.3}. Proposition \ref{prop.2} and \cite[Theorem 3.1]{B-R-1994} together conclude the proof of (2). By the proof of Proposition \ref{prop.1}, we conclude (3). And (4) follows from Proposition \ref{prop.4} and Corollary \ref{corollary-1} together. Proposition \ref{prop.1}(4) implies (5). (6) follow from (1) and Proposition \ref{prop.1}(3).
  \end{proof}
 \end{theorem}
 The next we construct the projective bundle over $\mathcal{M}_\alpha^s (t)$. Before going the construction,
first note that for each $i \in V$, let $(r_i, d_i) = 1$, there is a Poincar\'e $\alpha$-semistable quiver bundle $ \mathcal{P}_\bullet =  (\mathcal{P}_i )_{i \in V}$ on $\mathcal{M}_\alpha^s (t) \times C$ such that ${\mathcal{P}_i}_{\vert_{[E_i] \times C}} \cong [E_i]$. For general $r_i$ and $d_i$, one can prove that locally in the \'etale topology on $\mathcal{M}_\alpha^s(t)$, there is a Poincar\'e family in $\mathcal{M}_\alpha^s (t)$.

  \begin{theorem}\label{theorem construstion of poincare bundle}
   Let $Q = (V, A)$ be a finite quiver and let $C$ be a complex smooth projective curve. Let $\alpha = (\alpha_i)_{i \in V}$ be a stability parameter. Then for a fixed type $t = (r, d)$ on $Q$ where $r=(r_i)_{i \in V}$ and $d = (d_i)_{i\in V}$ with $r_i$ and $d_i$ are co-prime and for all $i \in V$, there is a locally non trivial map
   \[
   \pi : \mathbb{P} \rightarrow \mathcal{M}_\alpha^s(t)
   \]
   where $\mathbb{P}$ is a fiber product of projective bundles such that the fiber of $\pi$ at  a point $E_\bullet = {(E)}_{i \in V}$ is $\prod\limits_{a \in A} \mathbb{P}(\mathcal{E}_a)$.
   \begin{proof}
   Since $r_i$ and $d_i$ is co-prime for each $i \in V$, there is a Poincar\'e $Q$-bundles $\mathcal{P_\bullet} = {(\mathcal{P}_i)}_{i \in Q}$ on $\mathcal{M}_\alpha^s (t) \times C$ such that ${\mathcal{P}_i}_{\vert_{[E_i] \times C}} \cong [E_i]$ via the identification $[E_i]\times C \cong C$. Consider $
   \mathcal{U}_{C}(t) = \{ E_\bullet = (E_i)_{i \in V} \in \mathcal{M}_\alpha^s(t) \; \vert \; \text{for each}\; i, \; E_i \; \text{is semistable vector bundle on}\; C \}$.

   Let $\iota : \mathcal{U}_{C}(t) \hookrightarrow  \mathcal{M}_\alpha^s(t)$ be the inclusion map.
   Note that the moduli space $\mathcal{U}_{C}(t)$ can be identified with $\prod\limits_{i \in V}  \mathcal{U}_{C}(r_i, d_i)$. Therefore, it is a smooth irreducible projective variety over $C$. Let 
   \[ \pi_i : \mathcal{U}_{C}(t) \rightarrow \mathcal{U}_C(r_i, d_i)
   \]
    be the $i$-th projection. We define the following sheaves on $\mathcal{M}_\alpha^s(t)$:
   \[
   \mathcal{E}_{a} = \mathcal{H}om_C(\pi^*(\iota^*(\mathcal{P}_{at})), \pi^*(\iota^*(\mathcal{P}_{ah})))
   \]
   for all $a \in A$. Let 
   \[
   \pi : \mathbb{P} \rightarrow \mathcal{M}_\alpha^s(t)   \]
   be denote the fiber product of $\mathbb{P}(\mathcal{E}_{a})$ over $\mathcal{M}_\alpha^s(t)$. By the construction, the fibre of $\pi$ at a point $E_\bullet = {(E)}_{i \in V}$ is $\prod\limits_{a \in A} \mathbb{P}(\mathcal{E}_a)$.
   \end{proof}
   \end{theorem}

 \end{document}